\documentclass{article}
\usepackage{amsfonts}
\usepackage{latexsym,amsmath,amsthm,amssymb,cite,enumerate}
\usepackage[sans]{dsfont}
\usepackage{mathpazo}
\usepackage{bbm}
\usepackage{color}
\usepackage[bookmarksnumbered, colorlinks, plainpages]{hyperref}

\hypersetup{citecolor={blue}}
\numberwithin{equation}{section}
\newtheorem{theorem}{Theorem}[section]
\newtheorem{lemma}[theorem]{Lemma}
\newtheorem{corollary}[theorem]{Corollary}

\theoremstyle{definition}
\newtheorem{definition}[theorem]{Definition}
\newtheorem{example}[theorem]{Example}
\newtheorem{remark}[theorem]{Remark}

\begin{document}

\begin{center}
{\Large \textbf{Uniqueness  and  Ulam-Hyers-Mittag-Leffler stability results for the delayed fractional multi-terms differential equation involving the $\Phi$--Caputo  fractional Derivative}%
}

\vskip.30in
{\large \textbf{Choukri Derbazi $^{1,*}$ \,and \, Zidane Baitiche$^1$} }  \\[2mm]
{\small $^{1}$Department of Mathematics, Faculty of Exact Sciences, University Fr\`eres Mentouri Constantine 25000, Algeria.\\ Email: choukriedp@yahoo.com, \;baitichezidane19@gmail.com
}
\end{center}

\vskip 4mm

\noindent {\small \textbf{Abstract.}
The principal aim of the present paper is to establish the uniqueness and Ulam-Hyers Mittag-Leffler (UHML) stability of solutions for a new class of multi-terms fractional time-delay differential equations in the context of the $\Phi$-Caputo fractional derivative. To achieve this purpose, the generalized Laplace transform method alongside facet with properties of the Mittag-Leffler functions (M-LFs), are utilized to give a new representation formula of the solutions for the aforementioned problem. Besides that, the uniqueness of the solutions of the considered problem is also proved by applying the well-known Banach contraction principle coupled with the  $\Phi$-fractional Bielecki-type norm.   While the $\Phi$-fractional Gronwall type inequality and the Picard operator (PO)  technique combined with abstract Gronwall lemma are used to prove the UHML stability results for the proposed problem. Lastly,  an example is offered to assure the validity of the obtained theoretical results.
}
\vskip.15in \footnotetext{\textbf{AMS 2010 Mathematics Subject Classification}
: 34A08;  26A33; 34D10; 45N05.} \footnotetext{\textbf{\ Keywords}:
$\Phi $--Caputo fractional  derivative,  time-delay differential equations, generalized Laplace transform, uniqueness, Ulam-Hyers Mittag-Leffler (UHML) stability, $\Phi$-fractional Bielecki-type norm,  $\Phi$-fractional Gronwall type inequality.}

\footnotetext{\textbf{$^{*}$ Corresponding author: email}:
choukriedp@yahoo.com.} \baselineskip=14pt

\setcounter{section}{0} \numberwithin{equation}{section}
\section{Introduction}
In recent years,  the subject of fractional differential equations (FDEs) has aroused considerable attention among scientists. The mentioned branch has many applications in diverse disciplines of sciences and engineering.  For more details, see \cite{Hilfer,b0,Pud,Tarasov1}. In 2017, a novel fractional derivative with respect to another function $\Phi$  was formulated by Almeida. This new operator is called the $\Phi$-Caputo fractional derivative. In a subsequent year, the concerned derivative was generalized by Sousa and was named the $\Phi$-Hilfer fractional derivative. These novel operators can be regarded as a unification of some well known fractional operators in the literature.  For a detailed discussion on the basic theory of $\Phi$-Caputo and  $\Phi$-Hilfer fractional derivatives, we refer to the recent papers  \cite{Almeida1,SousaPsi} and the references therein. Currently, many mathematicians have addressed the existence and uniqueness of solutions as well as different types of  Ulam\rq{}s  stabilities of nonlinear FDEs involving various categories of fractional derivatives with the help of fixed point theory.  For details, we refer the reader to \cite{AlmeidaM,AlmeidaRE,Kharade,Kucche,Kucche1,
Kucche2,Liu,Mali,SousaPsiBE,
VanterlerB,Sousaimpulsive,SousaUlam,J.WangF,WangMGS}.

To the best of our knowledge, the uniqueness and the UHML stability of solutions for the delayed fractional multi-terms differential equation involving the $\Phi$--Caputo fractional derivative is not yet investigated. So getting motivation from the results proved in \cite{Otrocol,WangMGS}, in this paper, we mainly focus on the  uniqueness  and  the UHML stability of  solutions   for the following $\Phi$--Caputo fractional multi-terms differential equation ($\Phi$--Caputo FMTDE) with a finite delay of the form 
\begin{equation}
\begin{cases}
{^{c}\mathbb{D}}_{\mathfrak {m}^{+}}^{\mu ;\Phi }\mathfrak {z} (\ell )+\varrho\, {^{c}\mathbb{D}}_{\mathfrak {m}^{+}}^{\kappa ;\Phi }\mathfrak {z} (\ell )=\mathbb{Q}\bigl(\ell ,\mathfrak {z}
(\ell ), \mathfrak {z}(\mathfrak {f}
(\ell ))\bigr),\;\ell \in \Omega:=[\mathfrak {m} ,\mathfrak {n}], \\
\mathfrak {z}(\ell)=\alpha(\ell), \; \ell\in  [\mathfrak {m} -\sigma, \mathfrak {m}],
\end{cases}
\label{aa}
\end{equation}
where ${^{c}\mathbb{D}}_{\mathfrak {m}^{+}}^{\mu ;\Phi }$ and ${^{c}\mathbb{D}}_{\mathfrak {m}^{+}}^{\kappa ;\Phi }$ denote the $\Phi$-Caputo fractional derivatives, with the orders $\mu$ and $\kappa$ respectively such that $0<\kappa<\mu\leq1, \varrho, \sigma >0$, $ \mathbb Q\in C(\Omega\times\mathbb R^2, \mathbb R), \mathfrak {f} \in C(\Omega, [\mathfrak {m} -\sigma, \mathfrak {n}]), \mathfrak {f}
(\ell )\leq \ell, \alpha \in C(\ [\mathfrak {m} -\sigma, \mathfrak {m}], \mathbb R)$ and $\mathfrak {m}, \mathfrak {n} \in \mathbb R^{+}$, such that $\mathfrak {m}<\mathfrak {n}$.

It is worth noting that the results obtained in this paper are generalizations and partial continuation of some results obtained in \cite{AlmeidaRE,Liu1,Otrocol,Peng,C.Wang,WangMGS}. 

The outline of the paper is as follows. In Section \ref{SE2} we introduce some basic concepts needed throughout this paper.  Section \ref{Sec3} is devoted to establishing the main results in which the  uniqueness of the solutions  for the problem \eqref{aa} can be obtained under  the famous Banach\rq{}s Banach fixed point theorem along with the  $\Phi$-fractional Bielecki-type norm.  Then, we present the UHML stability result of the problem \eqref{aa} in Section \ref{Sec4}. Finally, in Section \ref{Sec5}, we give an example to verify our main result.
\section{Basic concepts}\label{SE2}
In this section, we introduce some necessary definitions and preliminary facts which will be used throughout this paper.

Let us consider on  the space  $\mathcal {X}:=C(\Omega, \mathbb R)$ the $\Phi$-fractional Bielecki-type norm  $\|\cdot\|_{\mathfrak C, \mathfrak{B}, \mu} $ given by previous studies \cite{SousaPsiBE,VanterlerB}
and defined by
\begin{align*}
\|\mathfrak{z}\|_{\mathcal X, \mathfrak{B}, \mu}:=\sup_{\ell \in  [\mathfrak {m}, \mathfrak {n}]}\frac{|\mathfrak{z}(\ell)|}{\mathbb M_{\mu}
\bigl(\beta(\Phi(\ell)-\Phi(\mathfrak {m}))^{\mu}\bigr)}, \quad \beta>0.
\end{align*}
Then,  $\bigl(\mathcal X, \|\cdot\|_{\mathcal X, \mathfrak{B}, \mu}\bigr)$ is a Banach space. In addition, let us denote by  $\mathcal {Y}:= C( [\mathfrak {m} -\sigma, \mathfrak {n}], \mathbb R)$   the   Banach space of  all continuous functions $y$ from $[\mathfrak {m} -\sigma, \mathfrak {n}]$  into $  \mathbb R$  equipped with the norm
\begin{align}\label{Bielecki}
\|\mathfrak{z}\|_{\mathcal Y, \mathfrak{B}, \mu}:=\sup_{\ell \in  [\mathfrak {m} -\sigma, \mathfrak {n}]}\frac{|\mathfrak{z}(\ell)|}{\mathbb M_{\mu}
\bigl(\beta(\Phi(\ell)-\Phi(\mathfrak {m}))^{\mu}\bigr)}, \quad \beta>0.
\end{align}
It's clear that  $
\|\mathfrak{z}\|_{\mathcal X, \mathfrak{B}, \mu}\leq \|\mathfrak{z}\|_{\mathcal Y, \mathfrak{B}, \mu}$.

Now, we recall the definition of the Mittag--Leffler functions (M-LFs).
\begin{definition}[\protect\cite{Gorenflo}] For $\mathrm p ,\mathrm q >0$ and $\;\vartheta\in \mathbb{R}$, the M-LFs of one and two parameters are given by
\begin{equation}\label{mgf}
\mathbb{M}_{\mathrm p }(\vartheta)=\sum_{k=0}^{\infty }\frac{\vartheta^{k}}{\Gamma (\mathrm p k+1)},\quad \mathbb{M}_{\mathrm p ,\mathrm q}(\vartheta)=\sum_{k=0}^{\infty }\frac{\vartheta^{k}}{\Gamma (\mathrm p k+\mathrm q
)}.
\end{equation}%
Clearly, $\mathbb{M}_{\mathrm p,1}(\vartheta )=\mathbb{M}_{\mathrm p}(\vartheta)$.
\end{definition}

\begin{lemma}[\protect\cite{Gorenflo,J.WangF}]
\label{PMitagg} Let $\mathrm p\in (0, 1) ,\mathrm q  >\mathrm p$ be arbitrary and $\vartheta \in \mathbb{R}$. The  functions   $\mathbb{M}_{\mathrm p}, \mathbb{M}_{\mathrm p, \mathrm p}$ and $\mathbb{M}_{\mathrm p, \mathrm q}$ are nonnegative and have
the following properties:
\begin{enumerate}
\item $\mathbb{M}_{\mathrm p}(\vartheta)\leq 1, \mathbb{M}_{\mathrm p, \mathrm q}(\vartheta)\leq \frac{1}{%
\Gamma (\mathrm q)} ,$ \quad for any $\vartheta<0$,
\item $\mathbb{M}_{\mathrm p, \mathrm q}(\vartheta)=\vartheta\mathbb{M}_{\mathrm p, \mathrm p+\mathrm q}(\vartheta)+ \frac{1}{%
\Gamma (\mathrm q)},$ \quad $\mathrm p, \mathrm q>0, \vartheta \in \mathbb{R}.$
\end{enumerate}
\end{lemma}
Let $\Phi \colon
\Omega\longrightarrow \mathbb{R}$ be an increasing differentiable
function such that $\Phi ^{\prime }(\ell )\neq 0$, for all $\ell \in\Omega.$
\begin{definition}[\protect\cite{Almeida1,b0}]
\label{r} The R-L fractional integral of
order $\mu>0$ for an integrable function $\mathfrak {z} \colon\Omega%
\longrightarrow \mathbb{R}$ with respect to $\Phi $ is described by
\begin{equation*}
\mathbb{I}_{\mathfrak {m}^{+}}^{\mu ;\Phi }\mathfrak {z} (\ell )=\int_{\mathfrak {m}}^{\ell}\frac{\Phi
^{\prime }(\eta)(\Phi (\ell)-\Phi (\eta))^{\mu-1}}{\Gamma(\mu)}\mathfrak {z} (\eta )%
\mathrm{d}\eta,  \label{RLFP}
\end{equation*}
\end{definition}

where $\Gamma (\mu)=\int_{0}^{+\infty }\ell ^{\mu -1}e^{-\ell }\mathrm{d}\ell
,\quad \mu >0$ is the Gamma function.


\begin{definition}[\protect\cite{Almeida1}]
Let $\Phi ,\mathfrak {z} \in C^{n}(\Omega,\mathbb{R})$%
. The Caputo fractional derivative of $\mathfrak {z}$ of order $ n-1<\mu <n $ with respect to $\Phi $
is defined by
\begin{equation*}
{^{c}\mathbb{D}}_{\mathfrak {m}^{+}}^{\mu ;\Phi }\mathfrak {z} (\ell )=\mbox{ }\mathbb{I}%
_{\mathfrak {m}^{+}}^{n-\mu ;\Phi }\mathfrak {z} _{\Phi }^{[n]}(\ell ),
\end{equation*}%
where $n=[\mu]+1$ for $\mu \notin \mathbb{N}$, $n=\mu $ for $\mu \in
\mathbb{N}$, and $\mathfrak {z} _{\Phi }^{[n]}(\ell )=\left( \frac{\frac{d}{d\ell }}{\Phi
^{\prime }(\ell )}\right) ^{n}\mathfrak {z} (\ell ).$
\end{definition}
Some basic properties of the $\Phi$-fractional operators are listed in the following Lemma.
\begin{lemma}[\protect\cite{Almeida1,AlmeidaM}]
\label{LMA2} Let $\mu ,\kappa, \beta >0,$ and $\mathfrak {z} \in C(\Omega,\mathbb{R})$.
Then for each $\ell \in\Omega$,

\begin{enumerate}
\item ${^{c}\mathbb{D}}_{\mathfrak {m}^{+}}^{\mu ;\Phi }\mathbb{I}_{\mathfrak {m}^{+}}^{\mu ;\Phi
}\mathfrak {z} (\ell )=\mathfrak {z} (\ell )$,

\item $\mathbb{I}_{\mathfrak {m}^{+}}^{\mu ;\Phi }{^{c}\mathbb{D}}_{\mathfrak {m}^{+}}^{\mu ;\Phi
}\mathfrak {z} (\ell )=\mathfrak {z} (\ell )-\mathfrak {z} (\mathfrak {m}),\quad 0<\mu \leq 1$,

\item $\mathbb{I}_{\mathfrak {m}^{+}}^{\mu ;\Phi }(\Phi (\ell)-\Phi (\mathfrak {m}))^{\kappa -1}=\frac{%
\Gamma (\kappa )}{\Gamma (\kappa +\mu )}(\Phi (\ell)-\Phi (\mathfrak {m}))^{\kappa+\mu -1},$

\item ${^{c}\mathbb{D}}_{\mathfrak {m}^{+}}^{\mu ;\Phi }(\Phi (\ell)-\Phi (\mathfrak {m}))^{\kappa -1}=%
\frac{\Gamma (\kappa )}{\Gamma (\kappa -\mu )}(\Phi (\ell)-\Phi (\mathfrak {m}))^{\kappa-\mu -1},$
\item $ \mathbb  I_{\mathfrak {m}^{+}}^{\mu ;\Phi }\left(\mathbb M_{\mu}\bigl(\beta(\Phi(\ell)-\Phi (\mathfrak {m}))^{\mu}\right)=\frac{1}{\beta}\left(\mathbb M_{\mu}\bigl(\beta(\Phi(\ell)-\Phi (\mathfrak {m}))^{\mu}-1\right),$\\
in particular  if we take $\beta=1$  we can get the following estimate
\item $ \mathbb  I_{\mathfrak {m}^{+}}^{\mu ;\Phi }\left(\mathbb M_{\mu}\bigl((\Phi(\ell)-\Phi (\mathfrak {m}))^{\mu}\right)\leq \mathbb M_{\mu}\bigl((\Phi(\ell)-\Phi (\mathfrak {m})^{\mu}\bigr).$
\end{enumerate}
\end{lemma}

\begin{definition} [\cite{JaradLaplace}] A function $\mathfrak z : [\mathfrak {m}, \infty) \to\mathbb{R}$ is said to be of $\Phi(\ell)$-exponential order if there exist non-negative constants $c_1, c_2, \mathfrak {n}$ such that 
\[
|\mathfrak z(\ell)| \leq c_1e^{c_2(\Phi(\ell)-\Phi (\mathfrak {m}))}, \quad \ell \geq \mathfrak {n}.
\]
\end{definition}
\begin{definition}[\cite{JaradLaplace}]
Let $\mathfrak {z}, \Phi:[\mathfrak {m}, \infty) \rightarrow\mathbb{R}$ be real valued functions such that $\Phi(\ell)$ is continuous and $\Phi'(\ell) > 0$ on $[\mathfrak {m}, \infty)$. The generalized Laplace transform of $\mathfrak {z}$ is denoted by
\begin{align}\label{GLT}
\mathbb{L}_{\Phi} \bigl\lbrace {\mathfrak {z}(\ell)} \bigr\rbrace= \int_{\mathfrak {m}}^{ \infty} e^{-\lambda (\Phi(\ell)-\Phi (\mathfrak {m}) ) } \mathfrak {m}(\ell)\Phi '(\ell) \,{\mathrm d}\ell, \quad \mbox{for all}\; \lambda>0,
\end{align}
provided that the integral in \eqref{GLT} exists.
\end{definition}
\begin{definition}[\cite{JaradLaplace}]
Let $\mathfrak z_{1}$ and $\mathfrak z_{2}$ be two functions which are piecewise continuous at each interval $[\mathfrak {m}, \mathfrak {n}]$ and of exponential order. We define the generalized convolution of $\mathfrak z_{1}$ and $\mathfrak z_{2}$ by
\[
( \mathfrak z_{1} \ast_{\Phi}\mathfrak z_{2}) ( \ell ) = \int_{\mathfrak {m}}^{\ell}\Phi '(\eta) \mathfrak z_{1}( \eta) \mathfrak z_{2} \bigl( \Phi ^{-1} \bigl(\Phi(\ell)+\Phi (\mathfrak {m})- \Phi(\eta) \bigr) \bigr) \,d\eta.
\]
\end{definition}
\begin{lemma}[\cite{JaradLaplace}]
Let $\mathfrak z_{1}$ and $\mathfrak z_{2}$ be two functions which are piecewise continuous at each interval $[\mathfrak {m}, \mathfrak {n}]$ and of exponential order. Then
\[
\mathbb{L}_{\Phi} \bigl\lbrace {\mathfrak z_{1} \ast_{\Phi}\mathfrak z_{2}} \bigr\rbrace
=\mathbb{L}_{\Phi} \bigl\lbrace {\mathfrak z_{1} } \bigr\rbrace\mathbb{L}_{\Phi} \bigl\lbrace {\mathfrak z_{2}} \bigr\rbrace.
\]
\end{lemma}
In the following lemma, we present the generalized Laplace transforms of some elementary functions
as well as the generalized Laplace transforms of 
the generalized fractional integrals and derivatives.
\begin{lemma}[\cite{JaradLaplace}]
\label{LaplacePsi} The
following properties are satisfied:
\begin{enumerate}
\item  $\mathbb{L}_{\Phi} \bigl\lbrace {1} \bigr\rbrace=\frac{1}{\lambda}, $ \quad  $\lambda>0,$
\item $\mathbb{L}_{\Phi} \bigl\lbrace {(\Phi (\ell)-\Phi (\mathfrak {m}))^{\mathrm r-1}} \bigr\rbrace =\frac{\Gamma (\mathrm r)}{\lambda^{\mathrm r}},$ \quad  $\mathrm r,\lambda>0,$
\item $\mathbb{L}_{\Phi} \bigl\lbrace {\mathbb{M}_{\mathrm p}\bigl(\pm\varrho(\Phi(\ell)-\Phi (\mathfrak {m}))
^{\mathrm p}\bigr)} \bigr\rbrace = \frac{\lambda^{\mathrm p-1}}{\lambda^{\mathrm p}\mp\varrho},$ \quad  $\mathrm p>0$ and $\left\vert\frac{\varrho}{\lambda^{\mathrm p}}\right\vert<1$,
\item $\mathbb{L}_{\Phi} \bigl\lbrace {(\Phi (\ell)-\Phi (\mathfrak {m}))^{\mathrm q -1}\mathbb M_{\mathrm p,\mathrm q}\bigl(\pm\varrho(\Phi(\ell)-\Phi (\mathfrak {m}))^{\mathrm p}\bigr)} \bigr\rbrace = \frac{\lambda^{\mathrm p-\mathrm q}}{\lambda^{\mathrm p}\mp\varrho},$ \quad  $\mathrm p>0$ and $\left\vert\frac{\varrho}{\lambda^{\mathrm p}}\right\vert<1$,\item $\mathbb{L}_{\Phi} \bigl\lbrace {\mathbb{I}_{\mathfrak {m}^{+}}^{\mu ;\Phi }\mathfrak {z} (\ell )} \bigr\rbrace = \frac{\mathbb{L}_{\Phi} \bigl\lbrace {\mathfrak {z} (\ell )} \bigr\rbrace }{\lambda^{\mu}},$ \quad  $\mu,\lambda>0$,
\item $\mathbb{L}_{\Phi} \bigl\lbrace {{^{c}\mathbb{D}}_{\mathfrak {m}^{+}}^{\mu ;\Phi }\mathfrak {z} (\ell )} \bigr\rbrace = \lambda^{\mu}\mathbb{L}_{\Phi} \bigl\lbrace {\mathfrak {z} (\ell )} \bigr\rbrace-\lambda^{\mu-1}\mathfrak {z} (\mathfrak {m}),$ \quad  $0<\mu\leq1$ and $\lambda>0$,
\end{enumerate}
\end{lemma}
The following lemma is a generalization of Gronwall's inequality.
\begin{lemma}[\cite{Gronwall}] \label{Gronwall1}
Let $\Omega$ be the domain of the nonnegative integrable functions $\mathfrak c_{1}, \mathfrak c_{2}$. Also, $\mathfrak c_{3}$ be a continuous, nonnegative and nondecreasing function defined
on $\Omega$ and $\Phi\in  C^1(\Omega, \mathbb R_+)$ be an increasing function with the restriction
that $\Phi^{\prime}(\ell) \neq 0, \forall \ell \in \Omega$.
If
\[
\mathfrak c_{1}(\ell) \leq \mathfrak c_{2}(\ell) + \mathfrak c_{3}(\ell)\int_{\mathfrak {m}}^{\ell}\Phi
^{\prime }(\eta)(\Phi (\ell)-\Phi (\eta))^{\mu -1}\mathfrak c_{1}(\eta)\mathrm{d}\eta,\; \ell \in  \Omega.
\]
Then
\[
\mathfrak c_{1}(\ell) \leq  \mathfrak c_{2}(\ell) + \int_{\mathfrak {m}}^{\ell}\sum_{n=0}^{\infty}\frac{(\mathfrak c_{3}(\ell)\Gamma(\mu))^{n}}{\Gamma(n\mu)}\Phi
^{\prime }(\eta)(\Phi (\ell)-\Phi (\eta))^{n\mu -1}\mathfrak c_{2}(\eta)\mathrm{d}\eta,\; \ell \in \Omega.
\]
\end{lemma}
\begin{corollary}[\cite{Gronwall}] \label{Gronwall2}
Under the conditions of Lemma \ref{Gronwall1}, let $\mathfrak c_{2}$ be a nondecreasing function on $\Omega$. Then we get that
\begin{align}\label{GI2}
\mathfrak c_{1}(\ell) \leq  \mathfrak c_{2}(\ell)\mathbb M_{\mu}\left(\Gamma(\mu)\mathfrak c_{3}(\ell)\bigl(\Phi (\ell)-\Phi (\mathfrak {m})\bigr)^{\mu}\right),\; \ell \in \Omega.
\end{align}
\end{corollary}
\begin{definition}[\cite{Rus}] Let $(\mathbb E, \mathrm d)$ be a metric space. An operator $\mathbb S: \mathbb E\to \mathbb E$ is a Picard operator (PO) if there exists $ \mathfrak b^{\ast}\in \mathbb E$ such that
\begin{enumerate}
\item $\mathbb {F}_{\mathbb S}=\{\mathfrak b^{\ast}\}$ where $\mathbb {F}_{\mathbb S}= \{\mathfrak b \in \mathbb E:  \mathbb S(\mathfrak b) = \mathfrak b\}$ is the fixed point set of $\mathbb S$;
\item the sequence $\{\mathbb S^{j}(\mathfrak b_{0})\}_{j\in\mathbb N}$ converges to $\mathfrak b^{\ast}$ for all $\mathfrak b_{0}\in \mathbb E$.
\end{enumerate}
\end{definition}
\begin{lemma}[\cite{Rus2} abstract Gronwall lemma] \label{AGL}
Let $(\mathbb E, \mathrm d, \leq)$ be an ordered metric space and $\mathbb S: \mathbb E\to \mathbb E$ be an increasing PO. Then, for $\mathfrak b \in \mathbb E, \mathfrak b  \leq \mathbb S\mathfrak b$ implies $\mathfrak b \leq \mathfrak b^{\ast}$.
\end{lemma}
\section{Uniqueness  result for the problem \eqref{aa}}\label{Sec3}
Before going to our main results, we state the following special linear cases of the problem \eqref{aa}.
\begin{lemma}\label{LEU}For a given $h \in  C( \Omega,\mathbb R), 0<\kappa<\mu\leq1$  and  $\varrho>0$,  the linear  $\Phi$--Caputo FMTDE
\begin{equation}\label{al}
\begin{cases}
{^{c}\mathbb{D}}_{\mathfrak {m}^{+}}^{\mu ;\Phi }\mathfrak {z} (\ell )+\varrho\, {^{c}\mathbb{D}}_{\mathfrak {m}^{+}}^{\kappa ;\Phi }\mathfrak {z} (\ell )=h(\ell),\;\ell \in  [\mathfrak {m}, \mathfrak {n}], \\
\mathfrak {z}(\ell)=\alpha(\ell), \;\ell\in [\mathfrak {m}- \sigma, \mathfrak {m}]
\end{cases}
\end{equation}
has a unique   solution given explicitly  as	\begin{align*}
 \mathfrak {z}(\ell)=\left\{
\begin{matrix}
\alpha(\mathfrak {m})+\int_{\mathfrak {m}}^{\ell}\mathbb W_{\Phi}^{\mu}(\ell, \eta)\mathbb M_{\mu-\kappa,\mu}\bigl(-\varrho(\Phi(\ell)-\Phi(\eta))^{\mu-\kappa}\bigr)h(\eta)\mathrm{d}\eta &,&\; \ell\in   [\mathfrak {m}, \mathfrak {n}],\\
\alpha(\ell)&,& \;\ell\in [\mathfrak {m}- \sigma, \mathfrak {m}],
\end{matrix}\right.
	\end{align*}
where
\[
\mathbb W_{\Phi}^{\mu}(\ell, \eta)=\Phi ^{\prime }(\eta)(\Phi (\ell)-\Phi (\eta))^{\mu -1}.
\]
\end{lemma}
\begin{proof}
Applying the generalized Laplace transform to both sides of the first equation of \eqref{al} and using Lemma \ref{LaplacePsi}, we obtain
	\begin{equation*}
\lambda^{\mu}\mathbb{L}_{\Phi} \bigl\lbrace {\mathfrak {z} (\ell )} \bigr\rbrace-\lambda^{\mu-1}\mathfrak {z} (\mathfrak {m})+\varrho\lambda^{\kappa}\mathbb{L}_{\Phi} \bigl\lbrace {\mathfrak {z} (\ell )} \bigr\rbrace-\varrho\lambda^{\kappa-1}\mathfrak {z} (\mathfrak {m})= \mathbb{L}_{\Phi} \bigl\lbrace {h (\ell )} \bigr\rbrace.	
 \end{equation*}
So,  
  \begin{align*}
\mathbb{L}_{\Phi} \bigl\lbrace {\mathfrak {z} (\ell )} \bigr\rbrace&=\varrho\frac{\lambda^{-1}}{\lambda^{\mu-\kappa}+\varrho} \alpha(\mathfrak{m})+\frac{\lambda^{\mu-\kappa-1}}{\lambda^{\mu-\kappa}+\varrho} \alpha(\mathfrak{m})+\frac{\lambda^{-\kappa}}{\lambda^{\mu-\kappa}+\varrho}\mathbb{L}_{\Phi} \bigl\lbrace {h (\ell )} \bigr\rbrace\\
&=\varrho\mathbb{L}_{\Phi} \bigl\lbrace {(\Phi (\ell)-\Phi (\mathfrak{m}))^{\mu-\kappa}\mathbb M_{\mu-\kappa,\mu-\kappa+1}\bigl(-\varrho(\Phi(\ell)-\Phi(\mathfrak {m}))^{\mu-\kappa}\bigr)} \bigr\rbrace	 \alpha(\mathfrak{m})\\
&+\mathbb{L}_{\Phi} \bigl\lbrace {\mathbb{M}_{\mu-\kappa}\bigl(-\varrho(\Phi(\ell)-\Phi(\mathfrak {m}))
^{\mu-\kappa}\bigr)} \bigr\rbrace \alpha(\mathfrak{m})\\
&+\mathbb{L}_{\Phi} \bigl\lbrace {(\Phi (\ell)-\Phi (\mathfrak {m}))^{\mu-1}\mathbb M_{\mu-\kappa,\mu}\bigl(-\varrho(\Phi(\ell)-\Phi(\mathfrak {m}))^{\mu-\kappa}\bigr)} \bigr\rbrace\mathbb{L}_{\Phi} \bigl\lbrace {h (\ell )} \bigr\rbrace.
 \end{align*}
 Taking the inverse generalized Laplace transform  to both sides of the last expression, we get 
 \begin{align*}
\mathfrak {z} (\ell )&=\left(\mathbb M_{\mu-\kappa}\bigl(-\varrho(\Phi(\ell)-\Phi(\mathfrak {m}))^{\mu-\kappa}\bigr)+\varrho(\Phi (\ell)-\Phi (\mathfrak {m}))^{\mu-\kappa}\mathbb M_{\mu-\kappa,\mu-\kappa+1}\bigl(-\varrho(\Phi(\ell)-\Phi(\mathfrak {m}))^{\mu-\kappa}\bigr)
\right) \alpha(\mathfrak{m})\\
&+h(\ell )\ast_{\Phi}(\Phi (\ell)-\Phi (\mathfrak {m}))^{\mu-1}\mathbb M_{\mu-\kappa,\mu}\bigl(-\varrho(\Phi(\ell)-\Phi(\mathfrak {m}))^{\mu-\kappa}\bigr)\\
&= \alpha(\mathfrak{m})+\int_{\mathfrak {m}}^{\ell}\Phi ^{\prime }(\eta)(\Phi (\ell)-\Phi (\eta))^{\mu -1}\mathbb M_{\mu-\kappa,\mu}\bigl(-\varrho(\Phi(\ell)-\Phi(\eta))^{\mu-\kappa}\bigr)
h(\eta)\mathrm{d}\eta,\; \ell \in  [\mathfrak {m}, \mathfrak {n}].
 \end{align*}
This ends the proof of Lemma \ref{LEU}.
\end{proof} 
As a result of Lemma \ref{LEU}, the problem \eqref{aa}  can be  converted to an  integral equation which takes the following
 form 
 \begin{equation}\label{operaQ}
 \mathfrak {z}(\ell)=\left\{
\begin{matrix}
\alpha(\mathfrak {m})+\int_{\mathfrak {m}}^{\ell}\mathbb W_{\Phi}^{\mu}(\ell, \eta)\mathbb M_{\mu-\kappa,\mu}\bigl(-\varrho(\Phi(\ell)-\Phi(\eta))^{\mu-\kappa}\bigr)\\
\times\mathbb{Q}\bigl(\eta ,\mathfrak {z}
(\eta ), \mathfrak {z}(\mathfrak {f}
(\eta ))\bigr)\mathrm{d}\eta &,&\; \ell\in   [\mathfrak {m}, \mathfrak {n}],\\
\alpha(\ell)&,& \;\ell\in [\mathfrak {m}- \sigma, \mathfrak {m}].
\end{matrix}\right.
   \end{equation}
We are now in position to present and prove our main results.
\begin{theorem}\label{banach} Assume that the following statements are valid:
\begin{enumerate}[(H1)]
\item The function $\mathfrak {f}: \Omega  \longrightarrow [\mathfrak {m} -\sigma, \mathfrak {n}] $  is continuous function with $ \mathfrak {f}
(\ell )\leq \ell$.  
\item The function $\mathbb Q: \Omega\times\mathbb R^2  \longrightarrow \mathbb R $  is continuous and there exist $\mathbb{L}_{\mathbb{Q}}>0$ such that
\[
|\mathbb{Q}(\ell, \mathfrak {b}_2, \mathfrak {a}_2) -\mathbb{Q}(\ell, \mathfrak {b}_1, \mathfrak {a}_1)| \leq \mathbb{L}_{\mathbb{Q}}\bigl(|\mathfrak {b}_2 -\mathfrak {b}_1|+|\mathfrak {a}_2 -\mathfrak {a}_1|\bigr),\quad  \ell \in \Omega , \mathfrak {a}_1,  \mathfrak {a}_2, \mathfrak {b}_1, \mathfrak {b}_2\in  \mathbb R.
\]
\end{enumerate}
Then the problem \eqref{aa}  possesses a unique solution which belong to the space $\mathcal {Y}\cap\mathcal {X}$.

\end{theorem}
\begin{proof}Transform the integral representation \eqref{operaQ} of the problem \eqref{aa} into a fixed point problem as follows:
\[
\mathfrak {z}= \mathbb P\mathfrak {z}, \quad \mathfrak {z} \in  \mathcal  {Y},
\]
where $\mathbb  P: \mathcal  {Y} \longrightarrow  \mathcal  {Y}$  is defined by
 \begin{equation}\label{opera}
 \mathbb  P\mathfrak {z}(\ell)=\left\{
\begin{matrix}
\alpha(\mathfrak {m})+\int_{\mathfrak {m}}^{\ell}\mathbb W_{\Phi}^{\mu}(\ell, \eta)\mathbb M_{\mu-\kappa,\mu}\bigl(-\varrho(\Phi(\ell)-\Phi(\eta))^{\mu-\kappa}\bigr)\\
\times\mathbb{Q}\bigl(\eta ,\mathfrak {z}
(\eta ), \mathfrak {z}(\mathfrak {f}
(\eta ))\bigr)\mathrm{d}\eta &,&\; \ell\in   [\mathfrak {m}, \mathfrak {n}],\\
\alpha(\ell)&,& \;\ell\in [\mathfrak {m}- \sigma, \mathfrak {m}].
\end{matrix}\right.
   \end{equation}

Clearly, the operator $\mathbb P$  is well-defined.  Moreover, the existence of a fixed point for the operator $\mathbb P$ will ensure the existence of the solution of the problem \eqref{aa}.  Our aim is to check  that $\mathbb P$ is a contraction operator with respect to the $\Phi$-fractional Bielecki-type norm.  Note that by definition of operator $\mathbb P$, for any $\mathfrak {z}_{1}, \mathfrak {z}_{2} \in \mathcal Y$ we have
\begin{equation*}
|\mathbb P \mathfrak {z}_{2}(\ell)-\mathbb P \mathfrak {z}_{1}(\ell)|=0, \quad \mbox{for all}\; \ell \in [\mathfrak {m} -\sigma, \mathfrak {m}].
 \end{equation*}
 On the other hand, keeping in mind the definition of  the operator  $\mathbb P$ on $ [\mathfrak {m}, \mathfrak {n}]$ together with  assumptions  {\rm(H1)}, (\rm{H2}) and Lemmas \ref{PMitagg},\ref{LMA2} we  can get
 \begin{align*}
|\mathbb P \mathfrak {z}_{2}(\ell)-\mathbb P \mathfrak {z}_{1}(\ell)|
&\leq2\mathbb L_{\mathbb Q}\| \mathfrak {z}_{2}-\mathfrak {z}_{1}\|_{\mathcal X, \mathfrak B, \mu}\int_{\mathfrak {m}}^{\ell}\frac{\mathbb W_{\Phi}^{\mu}(\ell, \eta)}{\Gamma(\mu)}\mathbb M_{\mu}
\bigl(\beta(\Phi(\eta)-\Phi(\mathfrak {m}))^{\mu}\bigr)\mathrm{d\eta}\\
&\leq \frac{2\mathbb L_{\mathbb Q}}{\beta}\left[\mathbb M_{\mu}
\bigl(\beta(\Phi(\ell)-\Phi(\mathfrak {m}))^{\mu}\bigr)-1\right]\|\mathfrak {z}_{2}-\mathfrak {z}_{1}\|_{\mathcal X, \mathfrak B, \mu}.
 \end{align*}
 Hence, the above inequality yields
 \begin{align*}
\|\mathbb P \mathfrak {z}_{2}-\mathbb P \mathfrak {z}_{1}\|_{\mathcal X, \mathfrak B, \mu}
&\leq \frac{2\mathbb L_{\mathbb Q}}{\beta}\| \mathfrak {z}_{2}-\mathfrak {z}_{1}\|_{\mathcal X, \mathfrak B, \mu}.
 \end{align*}
Thus
\begin{align*}
\|\mathbb P \mathfrak {z}_{2}-\mathbb P \mathfrak {z}_{1}\|_{\mathcal Y, \mathfrak B, \mu}
&\leq \frac{2\mathbb L_{\mathbb Q}}{\beta}\| \mathfrak {z}_{2}-\mathfrak {z}_{1}\|_{\mathcal Y, \mathfrak B, \mu}.
 \end{align*}
Let us choose $\beta>0$ such that $\frac{2\mathbb L_{\mathbb Q}}{\beta}<1$. It is easy to see that  the operator $\mathbb P$ is a contraction with respect to Bielecki's norm $\| \cdot \|_{\mathcal Y, \mathfrak B, \mu}$. Now, by applying the Banach's fixed point theorem, we can  find that $\mathbb P$ has a unique fxed point, and thus the problem \eqref{aa} has a unique  solution in the space $\mathcal {Y}\cap\mathcal {X}$. This completes the proof.
\end{proof}
\begin{remark}
 Notice that in our analysis we don\rq{}t assume that $\frac{2\mathbb L_{\mathbb Q}\left(\Phi(\mathfrak {n})-\Phi(\mathfrak {m})\right)^{\mu}}{\Gamma(\mu+1)}<1$ in Theorem \ref{banach}, while it is required in Theorem 3.4 in the article of Wang\ and\  Zhang \cite{WangMGS}.
\end{remark}
\section{ Ulam-Hyers-Mittag-Leffler stability results for the problem \eqref{aa}}\label{Sec4}
Motivated by \cite{UlamS,WangMGS}, we introduce the Ulam--Hyers--Mittag-Leffler stability
of solutions to our   problem \eqref{aa}.

 Let $\varepsilon, \varrho>0$ and $\zeta: \Omega\to{\Bbb {R}}^{+}$,  be a continuous function. We focus on the following inequality:
\begin{align}\label{S1T2}
\begin{aligned}&  \textstyle \bigl\vert {^{c}\mathbb{D}}_{\mathfrak {m}^{+}}^{\mu ;\Phi }\tilde{\mathfrak {z}} (\ell )+\varrho\, {^{c}\mathbb{D}}_{\mathfrak {m}^{+}}^{\kappa ;\Phi }\tilde{\mathfrak {z}} (\ell )-\mathbb{Q}\bigl(\ell ,\tilde{\mathfrak {z}}
(\ell ), \tilde{\mathfrak {z}}(\mathfrak {f}
(\ell ))\bigr)\bigr\vert \leq\varepsilon\mathbb M_{\mu}\bigl((\Phi(\ell)-\Phi (\mathfrak {m}))^{\mu}\bigr), \quad \displaystyle \ell \in \Omega.
 \end{aligned}	
\end{align}
\begin{definition}[\cite{WangMGS}]\label{DF1T2}
Equation  \eqref{aa} is UHML stable, with respect to $\mathbb M_{\mu}\bigl((\Phi(\ell)-\Phi (\mathfrak {m})^{\mu}\bigr)$ if there exists a real number $c_{\mathbb M_{\mu}}>0$ such that, for each $\varepsilon>0$ and for each solution $\tilde{\mathfrak {z}}\in \mathcal {Y} $ of the inequality \eqref{S1T2}, there is a unique solution solution $\mathfrak {z}\in \mathcal {Y} $ of  Eq. \eqref{aa}  with
\[
\begin{cases}\bigl\vert \tilde{\mathfrak {z}}(\ell)-\mathfrak {z}(\ell)\bigr\vert=0, \quad \ell\in [\mathfrak {m}- \sigma, \mathfrak {m}],\\
\bigl\vert \tilde{\mathfrak {z}}(\ell)-\mathfrak {z}(\ell)\bigr\vert \leq c_{\mathbb M_{\mu}}\varepsilon \mathbb M_{\mu}\bigl((\Phi(\ell)-\Phi(\mathfrak {m}))^{\mu}\bigr),\quad \ell \in [\mathfrak {m}, \mathfrak {n}].
\end{cases}
\]
\end{definition}

\begin{remark}[\cite{WangMGS}] \label{RSE1} A function $\tilde{\mathfrak {z}} \in \mathcal {X}$ is a solution of inequality \eqref{S1T2} if and only if there exists a function $\Theta \in\mathcal {X}$ (which depends on solution $\tilde{\mathfrak {z}}$) such that
\begin{enumerate}[(i)]
\item $|\Theta(\ell)| \leq \varepsilon \mathbb M_{\mu}\bigl((\Phi(\ell)-\Phi(\mathfrak {m}))^{\mu}\bigr), \quad \ell \in  \Omega$,
\item ${^{c}\mathbb{D}}_{\mathfrak {m}^{+}}^{\mu ;\Phi }\tilde{\mathfrak {z}} (\ell )+\varrho\, {^{c}\mathbb{D}}_{\mathfrak {m}^{+}}^{\kappa ;\Phi }\tilde{\mathfrak {z}} (\ell )=\mathbb{Q}\bigl(\ell ,\tilde{\mathfrak {z}}
(\ell ), \tilde{\mathfrak {z}}(\mathfrak {f}
(\ell ))\bigr)+\Theta(\ell), \quad \ell \in  \Omega$.
\end{enumerate}
\end{remark}
\begin{lemma}\label{LHUS}Let $\tilde{\mathfrak {z}} \in \mathcal {X}$ be a solution of of inequality \eqref{S1T2},  then $\tilde{\mathfrak {z}}$ satisfes the following integral inequality
  \begin{align*}
\begin{aligned}
&\left\vert\tilde{\mathfrak {z}}(\ell)-\tilde{\mathfrak {z}}(\mathfrak {m})-\int_{\mathfrak {m}}^{\ell}\mathbb W_{\Phi}^{\mu}(\ell, \eta)\mathbb M_{\mu-\kappa,\mu}\bigl(-\varrho(\Phi(\ell)-\Phi(\eta))^{\mu-\kappa}\bigr)\mathbb{Q}\bigl(\eta ,\tilde{\mathfrak {z}}
(\eta ), \tilde{\mathfrak {z}}(\mathfrak {f}
(\eta ))\bigr)\mathrm{d}\eta\right\vert\\
&\leq \varepsilon\mathbb M_{\mu}\bigl((\Phi(\ell)-\Phi(\mathfrak {m}))^{\mu}\bigr).
 \end{aligned}
 \end{align*}
\end{lemma}

\begin{proof}In fact, by the second part of Remark \ref{RSE1}, we have
\begin{align}\label{SOP1}
{^{c}\mathbb{D}}_{\mathfrak {m}^{+}}^{\mu ;\Phi }\tilde{\mathfrak {z}} (\ell )+\varrho\, {^{c}\mathbb{D}}_{\mathfrak {m}^{+}}^{\kappa ;\Phi }\tilde{\mathfrak {z}} (\ell )=\mathbb{Q}\bigl(\ell ,\tilde{\mathfrak {z}}
(\ell ), \tilde{\mathfrak {z}}(\mathfrak {f}
(\ell ))\bigr)+\Theta(\ell), \quad \ell \in  \Omega.	
\end{align}
   Thanks to  Lemma \ref{LEU}, the integral representation of \eqref{SOP1} is expressed as
\begin{align}\label{SPert}
\begin{aligned}
\tilde{\mathfrak {z}}(\ell)=&\tilde{\mathfrak {z}}(\mathfrak {m})+\int_{\mathfrak {m}}^{\ell}\mathbb W_{\Phi}^{\mu}(\ell, \eta)\mathbb M_{\mu-\kappa,\mu}\bigl(-\varrho(\Phi(\ell)-\Phi(\eta))^{\mu-\kappa}\bigr)\\
&\times\bigl\{\mathbb{Q}\bigl(\eta ,\tilde{\mathfrak {z}}
(\eta ), \tilde{\mathfrak {z}}(\mathfrak {f}
(\eta )\bigr)+\Theta(\eta)\bigr\}\mathrm{d}\eta.
\end{aligned}
\end{align}
 It follows from \eqref{SPert}, together with  the first part of Remark \ref{RSE1}, and Lemma \ref{PMitagg} that
 \begin{align*}
\begin{aligned}
&\left\vert\tilde{\mathfrak {z}}(\ell)-\tilde{\mathfrak {z}}(\mathfrak {m})-\int_{\mathfrak {m}}^{\ell}\mathbb W_{\Phi}^{\mu}(\ell, \eta)\mathbb M_{\mu-\kappa,\mu}\bigl(-\varrho(\Phi(\ell)-\Phi(\eta))^{\mu-\kappa}\bigr)\mathbb{Q}\bigl(\eta ,\tilde{\mathfrak {z}}
(\eta ), \tilde{\mathfrak {z}}(\mathfrak {f}
(\eta ))\bigr)\mathrm{d}\eta\right\vert\\
&\leq  \mathbb  I_{\mathfrak {m}^{+}}^{\mu ;\Phi }\Theta(\ell)\leq \varepsilon\mathbb  I_{\mathfrak {m}^{+}}^{\mu ;\Phi }\mathbb M_{\mu}\bigl((\Phi(\ell)-\Phi(\mathfrak {m}))^{\mu}\bigr).
 \end{aligned}
 \end{align*}
 Using  the sixth part of Lemma \ref{LMA2}, we can get
  \begin{align*}
\begin{aligned}
&\left\vert\tilde{\mathfrak {z}}(\ell)-\tilde{\mathfrak {z}}(\mathfrak {m})-\int_{\mathfrak {m}}^{\ell}\mathbb W_{\Phi}^{\mu}(\ell, \eta)\mathbb M_{\mu-\kappa,\mu}\bigl(-\varrho(\Phi(\ell)-\Phi(\eta))^{\mu-\kappa}\bigr)\mathbb{Q}\bigl(\eta ,\tilde{\mathfrak {z}}
(\eta ), \tilde{\mathfrak {z}}(\mathfrak {f}
(\eta ))\bigr)\mathrm{d}\eta\right\vert\\
&\leq \varepsilon\mathbb M_{\mu}\bigl((\Phi(\ell)-\Phi(\mathfrak {m}))^{\mu}\bigr).
 \end{aligned}
 \end{align*}

\end{proof}
Now, we discuss  the UHML stability of solutions 
 for  the problem \eqref{aa}. 
 
\begin{theorem}\label{USMT}
Under the assumptions of Theorem \ref{banach},  the problem \eqref{aa}is  UHML  stable.
\end{theorem}
\begin{proof}
Let $\varepsilon>0$ and let $\tilde{\mathfrak {z}} \in \mathcal {Y}\cap\mathcal {X}$ be a function which satisfies the inequality \eqref{S1T2}, and denote the unique solution of equation \eqref{aa} by  $\mathfrak {z} \in \mathcal {Y}\cap\mathcal {X}$, that is,
 \begin{equation*}
\begin{cases}
{^{c}\mathbb{D}}_{\mathfrak {m}^{+}}^{\mu ;\Phi }\mathfrak {z} (\ell )+\varrho\, {^{c}\mathbb{D}}_{\mathfrak {m}^{+}}^{\kappa ;\Phi }\mathfrak {z} (\ell )=\mathbb{Q}\bigl(\ell ,\mathfrak {z}
(\ell ), \mathfrak {z}(\mathfrak {f}
(\ell ))\bigr),\;\ell \in [\mathfrak {m}, \mathfrak {n}], \\
\mathfrak {z} (\ell )= \tilde{\mathfrak {z}}(\ell ),\; \ell\in  [\mathfrak {m} -\sigma, \mathfrak {m}].
\end{cases}
   \end{equation*}
By Theorem \ref{banach}, we have  
 \begin{equation*}
 \mathfrak {z}(\ell)=\left\{
\begin{matrix}
\tilde{\mathfrak {z}}(\mathfrak {m})+\int_{\mathfrak {m}}^{\ell}\mathbb W_{\Phi}^{\mu}(\ell, \eta)\mathbb M_{\mu-\kappa,\mu}\bigl(-\varrho(\Phi(\ell)-\Phi(\eta))^{\mu-\kappa}\bigr)\\
\times\bigl\{\mathbb{Q}\bigl(\eta ,\mathfrak {z}
(\eta ), \mathfrak {z}(\mathfrak {f}
(\eta ))\bigr)\bigr\}\mathrm{d}\eta &,&\; \ell\in   [\mathfrak {m}, \mathfrak {n}],\\
\tilde{\mathfrak {z}}(\ell )&,& \;\ell\in [\mathfrak {m}- \sigma, \mathfrak {m}].
\end{matrix}\right.
   \end{equation*}
Note that, when $ \ell\in [\mathfrak {m}- \sigma, \mathfrak {m}]$, we have
\begin{align*}
|\tilde{\mathfrak {z}}(\ell)-\mathfrak {z}(\ell)|=0.
\end{align*}
On the other side,  for each $\ell \in [\mathfrak {m}, \mathfrak {n}]$ we obtain
\begin{align*}
&|\tilde{\mathfrak {z}}(\ell)-\mathfrak {z}(\ell)|\leq 
\left\vert\tilde{\mathfrak {z}}(\ell)-\tilde{\mathfrak {z}}(\mathfrak {m})-\int_{\mathfrak {m}}^{\ell}\mathbb W_{\Phi}^{\mu}(\ell, \eta)\mathbb M_{\mu-\kappa,\mu}\bigl(-\varrho(\Phi(\ell)-\Phi(\eta))^{\mu-\kappa}\bigr)\mathbb{Q}\bigl(\eta ,\tilde{\mathfrak {z}}
(\eta ), \tilde{\mathfrak {z}}(\mathfrak {f}
(\eta ))\bigr)\mathrm{d}\eta\right\vert\\
&+ \int_{\mathfrak {m}}^{\ell}\mathbb W_{\Phi}^{\mu}(\ell, \eta)\mathbb M_{\mu-\kappa,\mu}\bigl(-\varrho(\Phi(\ell)-\Phi(\eta))^{\mu-\kappa}\bigr)\bigl|\mathbb{Q}\bigl(\eta ,\tilde{\mathfrak {z}}
(\eta ), \tilde{\mathfrak {z}}(\mathfrak {f}
(\eta ))\bigr)-\mathbb{Q}\bigl(\eta ,\mathfrak {z}
(\eta ), \mathfrak {z}(\mathfrak {f}
(\eta ))\bigr)\bigr|\mathrm{d}\eta.
\end{align*}
Using (H2) and Lemma \ref{LHUS}, we can arrive at
\begin{align}\label{OPRP1}
|\tilde{\mathfrak {z}}(\ell)-\mathfrak {z}(\ell)| \leq& \varepsilon\mathbb M_{\mu}\bigl((\Phi(\ell)-\Phi(\mathfrak {m}))^{\mu}\bigr)+\mathbb L_{\mathbb Q}\int_{\mathfrak {m}}^{\ell}\frac{\mathbb W_{\Phi}^{\mu}(\ell, \eta)}{\Gamma(\mu)}\bigl(|\tilde{\mathfrak {z}}(\eta)-\mathfrak {z}(\eta)|+|\tilde{\mathfrak {z}}(\mathfrak {f}
(\eta ))-\mathfrak {z}(\mathfrak {f}
(\eta ))|\bigr)\mathrm{d}\eta.
\end{align}
Now, for each $\mathfrak {b}\in C( [\mathfrak {m} -\sigma, \mathfrak {n}], \mathbb R_{+})$, we defne an operator $\mathbb {S}: C( [\mathfrak {m} -\sigma, \mathfrak {n}], \mathbb R_{+}) \to C( [\mathfrak {m} -\sigma, \mathfrak {n}], \mathbb R_{+})$ by
 \begin{equation}\label{operaS}
\mathbb {S} \mathfrak {b}(\ell)=
\begin{cases}
\varepsilon\mathbb M_{\mu}\bigl((\Phi(\ell)-\Phi(\mathfrak {m}))^{\mu}\bigr)+\mathbb L_{\mathbb Q}\int_{\mathfrak {m}}^{\ell}\frac{\mathbb W_{\Phi}^{\mu}(\ell, \eta)}{\Gamma(\mu)}\bigl(\mathfrak {b}(\eta)+\mathfrak {b}(\mathfrak {f}
(\eta ))\bigr)\mathrm{d}\eta,\; \ell\in   [\mathfrak {m}, \mathfrak {n}],\\
0, \quad \ell\in [\mathfrak {m}- \sigma, \mathfrak {m}].
\end{cases}
   \end{equation}
We prove that $\mathbb {S}$ is a Picard operator. Let $\mathfrak {b}_{1}, \mathfrak {b}_{2}\in C( [\mathfrak {m} -\sigma, \mathfrak {n}], \mathbb R_{+})$. Then,
\[
|\mathbb {S} \mathfrak {b}_{2}(\ell)-\mathbb {S} \mathfrak {b}_{1}(\ell)|=0, \;\ell\in [\mathfrak {m}- \sigma, \mathfrak {m}].
\]
Now, for any $\ell\in   [\mathfrak {m}, \mathfrak {n}]$, it follows from (H1) and (H2) that
 \begin{align*}
|\mathbb {S} \mathfrak {b}_{2}(\ell)-\mathbb {S} \mathfrak {b}_{1}(\ell)|
&\leq \frac{2\mathbb L_{\mathbb Q}}{\beta}\left[\mathbb M_{\mu}
\bigl(\beta(\Phi(\ell)-\Phi(\mathfrak {m}))^{\mu}\bigr)-1\right]\|\mathfrak {b}_{2}-\mathfrak {b}_{1}\|_{\mathcal X, \mathfrak B, \mu},
 \end{align*}
which leads to
 \begin{align*}
 \|\mathbb {S}\mathfrak {b}_{2}-\mathbb {S}\mathfrak {b}_{1}\|_{\mathcal Y, \mathfrak B, \mu}
&\leq \frac{2\mathbb L_{\mathbb Q}}{\beta}\|\mathfrak {b}_{2}-\mathfrak {b}_{1}\|_{\mathcal Y, \mathfrak B, \mu}.
 \end{align*}
Choosing $\beta>0$  such that $\frac{2\mathbb L_{\mathbb Q}}{\beta}<1$, we have that  $\mathbb S$ is a contraction with respect to Bielecki's norm $\| \cdot \|_{\mathcal Y, \mathfrak B, \mu}$. According to Banach fxed point theorem, we deduce that $\mathbb S$ is a Picard operator and $\mathbb F_{\mathbb S}=\{\mathfrak {b}^{\ast}\}$. Thus
 \begin{equation}\label{PSO}
\mathfrak {b}^{\ast}(\ell)=
\varepsilon\mathbb M_{\mu}\bigl((\Phi(\ell)-\Phi(\mathfrak {m}))^{\mu}\bigr)+\mathbb L_{\mathbb Q}\int_{\mathfrak {m}}^{\ell}\frac{\mathbb W_{\Phi}^{\mu}(\ell, \eta)}{\Gamma(\mu)}\bigl(\mathfrak {b}^{\ast}(\eta)+\mathfrak {b}^{\ast}(\mathfrak {f}
(\eta ))\bigr)\mathrm{d}\eta, \; \ell\in   [\mathfrak {m}, \mathfrak {n}].
   \end{equation}
   Next, we show that $\mathfrak {b}^{\ast}$ is increasing.  For this end, let Let any $\ell_{1}, \ell_{2}\in  [\mathfrak {m} -\sigma, \mathfrak {n}]$. If $\ell_{1}, \ell_{2}\in  [\mathfrak {m}-\sigma, \mathfrak {m}]$ with $\ell_{1} < \ell_{2}$, then $\mathfrak {b}^{\ast}(\ell_{2})- \mathfrak {b}^{\ast}(\ell_{1}) = 0$, and if $\ell_{1}, \ell_{2}\in  [\mathfrak {m}, \mathfrak {n}]$ provided that $\ell_{1} < \ell_{2}$. Denote $\nu= \min_{\eta\in  [\mathfrak {m}, \mathfrak {n}]}\bigl(\mathfrak {b}^{\ast}(\eta)+\mathfrak {b}^{\ast}(\mathfrak {f}
(\eta ))\bigr)$. Then,
 \begin{align*}
\mathfrak {b}^{\ast}(\ell_{2})-\mathfrak {b}^{\ast}(\ell_{1})&=
\varepsilon\left(\mathbb M_{\mu}\bigl((\Phi(\ell_{2})-\Phi(\mathfrak {m}))^{\mu}\bigr)-\mathbb M_{\mu}\bigl((\Phi(\ell_{1})-\Phi(\mathfrak {m}))^{\mu}\bigr)\right)\\
&+\mathbb L_{\mathbb Q}\int_{\mathfrak {m}}^{\ell_{1}}\frac{\mathbb W_{\Phi}^{\mu}(\ell_{2}, \eta)-\mathbb W_{\Phi}^{\mu}(\ell_{1}, \eta)}{\Gamma(\mu)}\bigl(\mathfrak {b}^{\ast}(\eta)+\mathfrak {b}^{\ast}(\mathfrak {f}
(\eta ))\bigr)\mathrm{d}\eta\\
&+\mathbb L_{\mathbb Q}\int_{\ell_{1}}^{\ell_{2}}\frac{\mathbb W_{\Phi}^{\mu}(\ell_{2}, \eta)}{\Gamma(\mu)}\bigl(\mathfrak {b}^{\ast}(\eta)+\mathfrak {b}^{\ast}(\mathfrak {f}
(\eta ))\bigr)\mathrm{d}\eta\\
&\geq \varepsilon\left(\mathbb M_{\mu}\bigl((\Phi(\ell_{2})-\Phi(\mathfrak {m}))^{\mu}\bigr)-\mathbb M_{\mu}\bigl((\Phi(\ell_{1})-\Phi(\mathfrak {m}))^{\mu}\bigr)\right)\\
&+\frac{\nu\mathbb L_{\mathbb Q}}{\Gamma(\mu+1)}\bigl((\Phi(\ell_{2})-\Phi(\mathfrak {m}))^{\mu}-(\Phi(\ell_{1})-\Phi(\mathfrak {m}))^{\mu}\bigr)\\
&>0.
   \end{align*}
This means that $\mathfrak {b}^{\ast}$ is  increasing. Keeping in mind (H1) we arrive  to 
$\mathfrak {b}^{\ast}(\mathfrak {f}(\ell ))\leq  \mathfrak {b}^{\ast}(\ell),\, \ell \in [\mathfrak {m}, \mathfrak {n}]$. Therefore, Eq. \eqref{PSO} reduces to
\begin{equation*}
\mathfrak {b}^{\ast}(\ell)\leq
\varepsilon\mathbb M_{\mu}\bigl((\Phi(\ell)-\Phi(\mathfrak {m}))^{\mu}\bigr)+\frac{2\mathbb L_{\mathbb Q}}{\Gamma(\mu)}\int_{\mathfrak {m}}^{\ell}\mathbb W_{\Phi}^{\mu}(\ell, \eta)\mathfrak {b}^{\ast}(\eta)\mathrm{d}\eta, \; \ell\in   [\mathfrak {m}, \mathfrak {n}].
   \end{equation*}
Applying Corollary \ref{Gronwall2} (the $\Phi$-fractional Gronwall's inequality Eq. \eqref{GI2}),  to above inequality with $\mathfrak {c}_1(\ell)=\mathfrak {b}^{\ast}(\ell),  \mathfrak {c}_2(\ell)= \varepsilon\mathbb M_{\mu}\bigl((\Phi(\ell)-\Phi(\mathfrak {m}))^{\mu}\bigr)$ and $\mathfrak {c}_3(\ell)=\frac{2\mathbb L_{\mathbb Q}}{\Gamma(\mu)}.$ Since  $\mathfrak {c}_2(\ell)$ is nondecreasing function on $\Omega$, we conclude
that
\begin{align}\label{GRS1}
\mathfrak {b}^{\ast}(\ell) &\leq  \varepsilon\mathbb M_{\mu}\left(2\mathbb {L}_{\mathbb Q}\bigl(\Phi (\ell)-\Phi (\mathfrak {m})\bigr)^{\mu}\right)\mathbb M_{\mu}\bigl((\Phi(\ell)-\Phi(\mathfrak {m}))^{\mu}\bigr)\nonumber\\
&\leq \varepsilon\mathbb M_{\mu}\left(2\mathbb {L}_{\mathbb Q}\bigl(\Phi (\mathfrak {n})-\Phi (\mathfrak {m})\bigr)^{\mu}\right)\mathbb M_{\mu}\bigl((\Phi(\ell)-\Phi(\mathfrak {m}))^{\mu}\bigr)\nonumber\\
&=c_{\mathbb M_{\mu}}\varepsilon\mathbb M_{\mu}\bigl((\Phi(\ell)-\Phi(\mathfrak {m}))^{\mu}\bigr),\quad \ell \in \Omega,
\end{align}
where $c_{\mathbb M_{\mu}}=\mathbb M_{\mu}\left(2\mathbb {L}_{\mathbb Q}\bigl(\Phi (\mathfrak {n})-\Phi (\mathfrak {m})\bigr)^{\mu}\right).$

In particular, if $\mathfrak {b}= |\tilde{\mathfrak {z}}-\mathfrak {z}|$, from \eqref{OPRP1}, $\mathfrak {b}(\ell)\leq  \mathbb {S}\mathfrak {b}(\ell)$ and applying the abstract Gronwall lemma (Lemma \ref{AGL}) we obtain $\mathfrak {b}(\ell) \leq \mathfrak {b}^{\ast}(\ell)$, where $\mathbb {S}$ is an increasing Picard operator. Combining this fact  with \eqref{GRS1}, it yelds that
\[
\begin{cases}\bigl\vert \tilde{\mathfrak {z}}(\ell)-\mathfrak {z}(\ell)\bigr\vert=0, \quad \ell\in [\mathfrak {m}- \sigma, \mathfrak {m}],\\
\bigl\vert \tilde{\mathfrak {z}}(\ell)-\mathfrak {z}(\ell)\bigr\vert \leq c_{\mathbb M_{\mu}}\varepsilon \mathbb M_{\mu}\bigl((\Phi(\ell)-\Phi(\mathfrak {m}))^{\mu}\bigr),\quad  \ell\in [\mathfrak {m}, \mathfrak {n}].
\end{cases}
\]
Thus, the problem \eqref{aa} is UHML stable. 
\end{proof}
\section{An Example}\label{Sec5}
In this fragment, we present an example where we apply both of  Theorems \ref{banach} and \ref{USMT} to some particular cases.\begin{example}Let us consider problem \eqref{aa} with specific data:
\begin{align}\label{data} 
\begin{aligned}
\mu=&0.5, \; \kappa=0.45,\; \varrho=1, \; \mathfrak {f}
(\ell )=\ell-\sigma.
\end{aligned}
\end{align} 
In order to illustrate Theorems \ref{banach} and \ref{USMT}, we take 
\begin{align}\label{dataf}
\mathbb Q(\ell, \mathfrak {z}(\ell), \mathfrak {z}(\mathfrak {f}
(\ell )))=&\frac{\sin \ell}{2}\left(\mathfrak {z}(\ell)+\sqrt{1+\mathfrak {z}^2(\ell)}\right) +\sin( \mathfrak {z}(\ell-\sigma)),
\end{align}
in \eqref{aa}. 
Obviously, the hypotheses ( \rm{H1}) and (\rm{H2}) hold with $\mathbb {L}_{\mathbb Q} =1$. It follows from Theorem \ref{banach} that the
problem \eqref{aa} with the data \eqref{data} and  \eqref{dataf}  has a unique solution  in $C( [\mathfrak {m} -\sigma, \mathfrak {n}], \mathbb R)\cap C( [\mathfrak {m}, \mathfrak {n}], \mathbb R).$ Also, by Theorem \ref{USMT}  the corresponding problem is UHML stable.
\end{example}
\begin{remark}  It is worth noting that in the previous example $\beta$ can be determined according to the assumptions of Theorem \ref{banach}. for example, we can choose $\beta=\mathbb {L}_{\mathbb Q}+1.$
\end{remark}
\end{document}